\documentclass[11pt]{article}

\usepackage{graphicx}
\usepackage[mathscr]{eucal}
\usepackage{amsmath,amsfonts,amsthm,amsopn,cite,mathrsfs}

\newcommand{\gabsys}{\mathcal{G}(g,\Lambda)}

\newcommand{\abZd}{a\mathbb{Z}^d \times b\mathbb{Z}^d}

\newcommand{\LiRt}{\textbf{\textit{L}}^1 (\mathbb{R}^{2})}
\newcommand{\LiRd}{\textbf{\textit{L}}^1 (\mathbb{R}^{d})}
\newcommand{\LiRtd}{\textbf{\textit{L}}^1 (\mathbb{R}^{2d})}
\newcommand{\Ltsp}{\textbf{\textit{L}}^2}
\newcommand{\LtRd}{\textbf{\textit{L}}^2 (\mathbb{R}^{d})}

\newcommand{\LtRtd}{\textbf{\textit{L}}^2 (\mathbb{R}^{2d})}
\newcommand{\LtR}{\textbf{\textit{L}}^2 (\mathbb{R})}

\newcommand{\LpRtd}{\textbf{\textit{L}}^p (\mathbb{R}^{2d})}
\newcommand{\Lpsp}{\textbf{\textit{L}}^p}

\newcommand{\Rd}{{\mathbb{R}^d}}
\newcommand{\Rtd}{\mathbb{R}^{2d}}
\newcommand{\RR}{\mathbb{R}}
\newcommand{\ZZ}{\mathbb{Z}}

\newcommand{\fthat}{\hat{\rule{0mm}{2.8mm}\;}}

\newcommand{\sch}{\mathcal{S}}
\newcommand{\vareps}{\varepsilon}
\newcommand{\supp}{\text{supp\,}}
\newcommand{\sinc}{\text{sinc}}

\theoremstyle{plain}
\newtheorem{theorem}{Theorem}[section]
\newtheorem{proposition}[theorem]{Proposition}
\newtheorem{lemma}[theorem]{Lemma}
\theoremstyle{definition}
\theoremstyle{remark}

\newcommand{\tfa}{time-frequency analysis}

\newcommand{\tf}{time-frequency}

\newcommand{\fif}{if and only if}
\newcommand{\tfs}{time-frequency shift}

\newcommand{\bdl}{band-limited}

\newcommand{\psdo}{pseudodifferential operator}

\newcommand{\beqa}{\begin{eqnarray*}}
\newcommand{\eeqa}{\end{eqnarray*}}

\newcommand{\field}[1]{\mathbb{#1}}
\newcommand{\bR}{\field{R}}        
\newcommand{\bZ}{\field{Z}}        
 \def\cS{\mathcal{S}}

\def\rd{\bR^d}

\def\zd{\bZ^d}

\def\rdd{{\bR^{2d}}}

\def\<{\left<}
\def\>{\right>}

\def\mv1{M_v^1}

\renewcommand{\qed}{\hfill \rule{7pt}{8pt} \vskip .2truein}   

\begin{document}

\title{Uniqueness and Reconstruction Theorems for Pseudodifferential
  Operators with a Bandlimited Kohn-Nirenberg Symbol
      \thanks{This work was funded by the Austrian Science Fund (FWF) in project NFN SISE S106.}
}


\author{Karlheinz Gr\"{o}chenig         \and
        Elmar Pauwels 
}




\maketitle

\begin{abstract}
Motivated by the problem of channel estimation in wireless
communications, we derive a  reconstruction formula for \psdo s with
a bandlimited  symbol. This reconstruction formula  
uses the diagonal entries of the matrix of the pseudodifferential
operator with respect to a Gabor system.  
In addition, we prove several other uniqueness theorems that shed
light on the relation between a \psdo\ and its matrix with respect to
a  Gabor system. 

\end{abstract}

\section{Introduction}
\label{sec:intro}

The mathematical formulation of orthogonal frequency
division multiplexing (OFDM) in wireless communications uses several
fundamental notions from  \tfa . On the one hand, Gabor expansions
are used to transform digital information into an analog signal. On the 
other hand, \psdo s are used to model the  distortion of a signal by 
the physical channel. Inevitably the rigorous analysis
of the communication system leads to new and interesting questions in
\tfa\ that are quite relevant for communication engineering. 

In this paper we study a  problem arising in  channel estimation.
Which  information is required to determine the 
symbol of a \psdo ? How can an operator be reconstructed from such
information?  

To put the discussion on a firm basis, let  us describe an extremely
simplified model of signals and the transmission in wireless
communications. See~\cite{Gr10,str06} for the mathematical models.  Let   $\pi (z)f(t) = e^{2\pi i
  z_2 t} f(t-z_1)$ denote the \tfs\ by $z=(z_1,z_2) \in \bR ^2$, and
  let $\Lambda = a\bZ \times b \bZ $ denote a lattice in the \tf\
  plane with lattice parameters $a,b>0$.  In
  orthogonal frequency division multiplexing (OFDM)  a string of numbers $c_\lambda $, the
  ``digital'' information, is used as the coefficient sequence for a Gabor
  series  of the form 
$$
f = \sum _{k\in \bZ  } \sum _{|l| \leq B} c_{kl} e^{2\pi i b l t}
  g(t-ak) = \sum _{\lambda \in \Lambda } c_\lambda \pi (\lambda )g \, , t\in \bR \, . 
$$

The pulse $g$ is
usually taken to be a characteristic function, but in nonstationary
environments pulses with better frequency concentration are
preferable \cite{boduhl99, grhahlmasvXX, hlmasc02, komo98, sc06}.
The analog  signal thus built is then transmitted from a sender to a
receiver and  distorted or transformed by physical processes. 

The second link between wireless communications and  \tfa\ is the description of the distortion of
the signal $f$ during the physical transmission. As a result of multipath
propagation  and of  the Doppler effect,  the received signal is a
superposition of time-frequency shifts. Specifically, the received 
signal can be  written as 
$$
\tilde f (t) = \iint \hat{\sigma } (\eta ,u ) \pi (-u,\eta) f(t)
\, dud\eta  \, .
$$
Here $\hat{\sigma } $ is the Fourier transform of a symbol $\sigma $
on $\mathbb{R}^2$ 
and is  called the \textit{spreading function} that indicates
the amplitude of each occuring \tfs . 
In the standard mathematical language the  distortion $f \to  \tilde
f$  is just the  \psdo\ (in the
Kohn-Nirenberg calculus) with symbol $\sigma $, and is usually written
as 
$$
\tilde f (t) = \sigma ^{KN} f(t) = \int \sigma (x,\xi ) e^{2\pi i x\xi
} \hat{f}(\xi )\, d\xi \, .
$$
For physical reasons the time delay and the Doppler shift must be
bounded, and therefore the spreading function $\hat{\sigma }$ has a
compact support. Equivalently, the symbol $\sigma $ is bandlimited,
i.e., analytic and of
exponential type. From the
perspective  of analysis,  such \psdo s are extremely special and are
only   the raw material for the study of difficult
operators~\cite{HorIII85}.  For wireless communications, \psdo s with
bandlimited  symbols are precisely the appropriate model.

At the receiver, the distorted analog signal $\tilde f$ is analyzed by
taking correlations with \tfs s of the given pulse (or some other
pulse). Thus the data to be analyzed are therefore the numbers 
\begin{equation}
  \label{eq:c4}
y_\lambda = \langle \tilde f , \pi (\lambda ) g\rangle = \sum _{\mu
  \in \Lambda }  c_\mu
\langle \sigma ^{KN} \pi (\mu )g, \pi (\lambda )g\rangle  \quad
\lambda \in \Lambda \, .
\end{equation}
The task of the engineer is now to recover and estimate the original
data $c_\lambda $ from the received information $y_\lambda $. 
The central object here is the matrix $H$ with entries  
\begin{equation}
  \label{eq:c5}
H_{\lambda \mu } =   \big(\langle \sigma ^{KN} \pi (\mu )g, \pi (\lambda )g\rangle
\big)_{\lambda , \mu \in \Lambda }
\end{equation}
 of the \psdo\ with respect to the
set of \tfs s $\{ \pi (\lambda ) g : \lambda \in \Lambda \}$. In
wireless communications this matrix is called the \emph{channel
matrix}. Its estimation and inversion are among  the principal
engineering tasks. 

A fundamental mathematical problem concerns the relation between the
channel matrix and the  symbol. This range of questions  has been
studied  in \tfa , e.g., in~\cite{gro06,Gr10}, yet the
models and assumptions of wireless communications pose new and
intriguing problems. An important
objective is to recover or approximate the symbol $\sigma $  from partial
information about the channel matrix; this is the problem of
\emph{channel estimation}. Usually, in  real  wireless communication
systems, pilot tones are used to estimate some entries of the channel
matrix on the diagonal   \cite{coerpuba02, best03}. The problem then is to recover the entire
matrix \eqref{eq:c5} and subsequently to solve the system $y=Hc$.

 The engineering models lead to the
mathematical question when and how the channel matrix is completely
determined by its diagonal. For arbitrary operators, this question
does not even  make
sense, but for operators with an  analytic symbol, as we will see,  one can recover the symbol completely
from the diagonal of the channel matrix. Our
first contribution is a precise reconstruction formula for the symbol
from the diagonal of the channel matrix. In other words, the matrix is
completely determined by its diagonal! The hypothesis that $\sigma $
is \bdl\ suggests a connection to the sampling theory of \bdl\
functions. Indeed, once this connection is established (which we 
do in Lemma \ref{lem:diagelem} below), one may apply results from the Shannon sampling
theory and obtains the following reconstruction formula. 

\begin{theorem}\label{firsta}
Let $\Lambda = a \zd \times b \zd$, $g\in \cS (\rd )$. If $\sigma \in
\cS ' 
(\rdd  )$, $\supp \hat{\sigma } \subseteq Q_\vareps = [-\frac{1}{2a} +
\vareps,\frac{1}{2a} - \vareps]^d \times [-\frac{1}{2b} +
\vareps,\frac{1}{2b} - \vareps]^d$ for some $\epsilon >0$, then there exists a kernel $K\in
\cS (\rdd )$ such that  
\begin{equation}
  \label{eq:c1}
\sigma = \sum _{\lambda \in \Lambda } \langle \sigma ^{KN} \pi
(\lambda )g, \pi (\lambda ) g\rangle T_\lambda K
  \end{equation}
with convergence in $\cS ' $.   
\end{theorem}

Theorem~\ref{firsta} provides  a theoretical answer to a crucial point of
channel estimation: how can the channel be estimated from (partial) 
knowledge of the diagonal of the channel matrix? 

We formulate  several versions  of this theorem that reflect various
models and assumptions in wireless communications. In engineering it
is usually  assumed that the channel is a Hilbert-Schmidt operator and 
$\sigma \in \LtRd$. In this case, the series \eqref{eq:c1}
converges in $\Ltsp$, and \eqref{eq:c1} is
valid under weaker assumptions on the pulse~$g$.  In our
opinion, however, the distributional version offers a better model for 
signal propagation since 
``point scatterers'' correspond  precisely to 
point measures in the symbol. Moreover, for the stable recovery of the
coefficients $c_\lambda $ from  \eqref{eq:c4}, the channel
matrix must be invertible, which is certainly not the case for a
Hilbert-Schmidt operator. 

Theorem~\ref{firsta} can be interpreted as a result about operator
identification and shares several  aspects with the work of Pfander et
al. \cite{kopf06, pf08}. In a series of papers, they investigate the question under
what conditions a symbol $\sigma $ can be recovered from a single
measurement $\sigma ^{KN}f$ on  a suitable distribution  $f$.  Similarly
to Theorem~\ref{firsta}, their answer is expressed as a sampling
theorem and is valid for bandlimited symbols. 

Our second contribution is the analysis of some common 
assumptions  in the engineering community.  A common assumption in the wireless communications 
literature is that the channel matrix is diagonal so that its
inversion becomes trivial. Equivalently, this means that   the
channel matrix is diagonalized by the \tfs s of a suitable function.  There are numerous  papers 
building on this assumption, e.g. 
\cite{dmbssbook10, boduscsh10, grhahlmasc07, shhwdaka10}. 
We show that this assumption cannot  withstand mathematical scrutiny. 

 We prove that if the underlying Gabor system is a frame for
 $\LtRd$, but not a basis, then  the corresponding
channel matrix $H$ cannot be  diagonal, unless  the operator $\sigma^{KN}$ is 
identically zero. We further  prove that for  a non-zero \psdo\ with a
bandlimited symbol 
and a  Gaussian window, the channel matrix cannot vanish identically, quite
independently of the spanning properties of the Gabor system. 

The paper is organized as follows: In Section \ref{sec:preliminaries}, we summarize the mathematical preliminaries. 
In Section \ref{sec:reconstruction} we present the reconstruction formula for the symbol of the pseudodifferential 
operator, and in Section \ref{sec:uniqueness} we present the proposed uniqueness results.

\section{Preliminaries}
\label{sec:preliminaries}

We collect some concepts and definitions from time-frequency
analysis. The precise  details and proofs can be found in \cite{folland89} or in \cite{gr01}. 

The \textit{Fourier transform} of a function $f \in \LtRd$ is defined as
\begin{equation}
  \hat{f}(\xi) = \int_{\RR^d} f(x) e^{-2 \pi i \xi \cdot x} dx, \quad \xi \in \RR^d. \nonumber
\end{equation}
The two fundamental operators in time-frequency analysis are the
\textit{translation operators} $T_x$  
and the \textit{modulation operators} $M_\xi$ defined by 
\begin{equation}
  T_x f(t) = f(t-x) \quad\text{and}\quad M_\xi f(t) = e^{2 \pi i \xi \cdot t} f(t), \quad t,x,\xi \in \RR^d. \nonumber
\end{equation}
Their compositions  are the  \textit{time-frequency shift operators } $\pi$ defined as
\begin{equation}
    \pi(z) f(t) = M_\xi T_x f(t) = e^{2 \pi i \xi \cdot t} f(t-x),
    \quad \text{ for }  z = (x,\xi) \in \Rtd. \nonumber
\end{equation}

A  set of  time-frequency shifts of a non-zero window function $g \in \LtRd$ with 
respect to a lattice $\Lambda = a\ZZ^d \times b\ZZ^d \subseteq \Rtd, a,b > 0$,
\begin{equation} \label{equ:gabsys}
  \gabsys  = \{ \pi(\lambda) g: \lambda \in \Lambda\} \nonumber
\end{equation}
is called a \textit{Gabor system}. If there exist constants
$ A,B > 0$ such that for all $f \in \LtRd $
\begin{equation}
  A \| f \|^2 \leq \sum_{\lambda \in \Lambda } | \langle f, \pi
  (\lambda ) g \rangle |^2 \leq B \| f \|^2 \, , \nonumber
\end{equation}
then the set $\gabsys $ is called a \emph{Gabor frame} with
\textit{frame bounds} $A$ and $B $.

The \textit{short-time Fourier transform} (STFT) of a function or
distribution $f$ with respect to a non-zero window $g$ is defined as
\begin{eqnarray} \label{equ:stft}
  V_g f(x,\xi) 
  &=& \int_\Rd f(t) \overline{ g(t-x) } e^{-2 \pi i \xi \cdot t} dt \nonumber \\
  &=& \langle f, M_\xi T_x g \rangle = \langle f, \pi(z) g \rangle, \nonumber
\end{eqnarray}
for $z = (x,\xi )  \in \RR^{2d}$.	

The \textit{Rihaczek distribution} of two functions $f,g \in \LtRd$ is defined as
\begin{equation} \label{equ:rihaczek}
  R(f,g)(x,\xi) = f(x) \overline{\hat{g}(\xi)} e^{-2 \pi i x \cdot \xi}, 
\end{equation}
for $x,\xi \in \RR^d$.	The Rihaczek distribution and the short-time Fourier transform are related in the following way:
\begin{equation} \label{equ:rihaczekstft}
  \widehat{R(f,g)} = \mathcal{U} V_g f,
\end{equation}
where $\mathcal{U} F(\xi,x) = F(-x,\xi)$ and $x,\xi \in \RR^d$.

Let $\sigma$ be a (measureable) function or a tempered  distribution on
$\Rtd$. The bilinear form  
\begin{equation}\label{equ:kn2}
  \langle \sigma^{KN} f,g \rangle = \langle \sigma, R(g,f) \rangle, \quad f,g \in \mathcal{S}(\RR^d),
\end{equation}
defines a linear operator $\sigma ^{KN}$  from $\mathcal{S} (\RR ^d)$ to $\mathcal{S}'
(\RR ^d)$. The operator $\sigma ^{KN}$ is a \psdo\ in the
Kohn-Nirenberg calculus with \textit{Kohn-Nirenberg symbol} $\sigma $. 
If $\hat{\sigma }$ is a locally integrable  function, then the
\textit{Kohn-Nirenberg transform} can also be written as 
\begin{equation} \label{equ:kn}
  \sigma^{KN} f(x) = \iint_{\Rtd} \hat{\sigma}(\eta,u) M_\eta T_{-u}
  f(x) \,du d\eta \, .
\end{equation}
In engineering this version is called the spreading representation of
the \psdo\ and  $\hat{\sigma}$ in (\ref{equ:kn}) is known as the \textit{spreading
  function} of $\sigma^{KN}$, since it describes how much the function
$f$ is "spread out" in time and frequency under the action of
$\sigma^{KN}$. 

\section{Reconstruction Formula}
\label{sec:reconstruction}

We first deal with the question  if and how the symbol of a \psdo\
$\sigma ^{KN}$ can be reconstructed from the diagonal of the channel
matrix. 

Let $\Lambda = a\ZZ^d \times b\ZZ^d$ be a lattice in $\Rtd$, $g \in
\LtRd $ a non-zero window function,  and
$\mathcal{G}( g, \Lambda )$ be the corresponding  Gabor system. 
 The matrix $H$ of a  pseudodifferential operator $\sigma^{KN}$ with
 respect to the Gabor system $\gabsys$ is defined  as follows:
\begin{equation}
  H_{\lambda \mu} = \langle \sigma^{KN} \pi(\mu)g, \pi(\lambda)g \rangle, \quad \lambda,\mu \in \Lambda. \nonumber
\end{equation}
When  $\sigma^{KN}$ describes a wireless channel, the matrix $H$
describes the action of a wireless channel on  
certain transmit pulses and is therefore called the \textit{channel
  matrix}.

We remark that for the solution of the linear
equation~\eqref{eq:c4} $H$  must be invertible. In wireless
communications it is costumary to assume that $\mathcal{G}(g, \Lambda )$
is a Riesz basis for the generated subspace~\cite{grhahlmasc07, hamasc06, komo98, st01}.

We first  derive an alternative  expression for the diagonal entries
of $H$ in terms of the Rihaczek distribution of $g$. Let 
$\mathcal{F} \LiRtd$ denote the Fourier algebra on $\Rtd$ consisting
of all functions on $\Rtd$  with integrable Fourier transform. 
  
We have the following well-known lemma. 
\begin{lemma}\label{lem:diagelem}
  Let $\sigma \in \LpRtd, 1 \leq p < \infty$, $\hat{\sigma}$ compactly
  supported, $g \in \LiRd \cap \mathcal{F} \LiRtd$. The diagonal entries of $H$ can be written as follows
  \begin{equation} \label{equ:diagelem}
    H_{\lambda, \lambda} = \langle \sigma^{KN} \pi(\lambda)g, \pi(\lambda)g \rangle = \sigma \ast R(g,g)^\ast (\lambda), 
    \quad \lambda \in \Lambda, \nonumber
  \end{equation}
where $f^\ast (x) = \overline{f(-x)}$.
\end{lemma}
\begin{proof}
Under the given assumptions, $\sigma$ is infinitely differentiable and $D^\alpha \sigma$ is bounded for all multi-indices $\alpha$.
The standard theory of pseudodifferential operators implies that
$\sigma^{KN}$ is bounded on
$\LtRd$~\cite{folland89,HorIII85}. Consequently  the mapping $\lambda
\in \Rtd \rightarrow \langle \sigma^{KN} \pi(\lambda)g,\pi(\lambda)g \rangle$ is
continuous, and  the 
channel matrix is well-defined.

Using the definition of the Rihaczek distribution (\ref{equ:rihaczek}), we get
\begin{equation} \label{equ:rhihaczekL2}
  \lVert R(g,g)^\ast \rVert_1 = \lVert g \otimes \hat{g} \rVert_1 = \lVert g \rVert_1 \cdot \lVert \hat{g} \rVert_1 < \infty. \nonumber
\end{equation}
Since $\sigma \in \LpRtd$, the convolution $\sigma \ast R(g,g)^\ast$ is well-defined in the $\Lpsp$-sense.
	
From the intertwining property of the Rihaczek distribution \cite[Lemma 4.2]{grst07}, we have
\begin{equation} \label{equ:rihaczekinterspec}
  R(\pi(\lambda)g,\pi(\lambda)g)(z) = R(g,g)(z-\lambda).
\end{equation}
Combining the definition of the Kohn-Nirenberg transform (\ref{equ:kn2}) and Equation (\ref{equ:rihaczekinterspec}), we obtain
\begin{eqnarray}
  \langle \sigma^{KN} \pi(\lambda)g, \pi(\lambda)g \rangle \nonumber
  &=& \langle \sigma, R(\pi(\lambda)g,\pi(\lambda)g) \rangle \nonumber \\
  &=& \int_{\RR^{2d}} \sigma(z) \overline{R(\pi(\lambda)g,\pi(\lambda)g)}(z) dz \nonumber \\
  &=& \int_{\RR^{2d}} \sigma(z) \overline{R(g,g)}(z-\lambda) dz \nonumber \\
  &=& \int_{\RR^{2d}} \sigma(z) R(g,g)^\ast (\lambda - z) dz \nonumber \\
  &=& \sigma \ast R(g,g)^\ast (\lambda). \label{equ:sigmari}
\end{eqnarray}
In general, (\ref{equ:sigmari}) is valid for almost every $\lambda \in \Rtd$. 
Since 
$$
  \supp \mathcal{F} \Big(\sigma \ast R(g,g)^\ast \Big)
  = \supp \Big(\hat{\sigma} \cdot \widehat{R(g,g)^\ast} \Big) 
  \subseteq \supp \hat{\sigma} \, ,
$$
is compact, $\sigma \ast R(g,g)^\ast$ is an analytic function, and therefore (\ref{equ:sigmari}) 
is valid for every $\lambda \in \Rtd$.
\end{proof}

By combining the observation  of Lemma \ref{lem:diagelem} with the
classical  sampling
theorem for band-limited functions, we obtain a  reconstruction
formula for the symbol of a \psdo\ from the diagonal of the channel
matrix. In the formulation of the multivariate version of the
Shannon-sampling theorem, we need the  ``sinc''-function adapted to
a lattice $\Lambda = a\ZZ ^d \times b\ZZ ^d \subseteq \Rtd$, namely
$$
\sinc (x) = \prod _{j=1}^d \frac{\sin \pi a x_j}{\pi a x_j}\prod
_{j=d+1}^{2d} \frac{\sin \pi b x_j}{\pi b x_j} \, .
$$
Then every function $\sigma \in \LtRtd $ with $\supp \hat{f} \subseteq 
Q = [-\frac{1}{2a},\frac{1}{2a}]^d \times
[-\frac{1}{2b},\frac{1}{2b}]^d$  possesses the cardinal series
expansion
\begin{equation}
  \label{eq:c7}
f = \sum _{\lambda \in \Lambda } f(\lambda ) T_\lambda \sinc   \nonumber
\end{equation}
 with convergence in $\Ltsp$ and uniformly. For the general theory of
Shannon sampling we refer to ~\cite{buspst88, ma91, pa97}. 

The first  reconstruction formula
is stated for symbols  $\sigma $  in $ \LpRtd, 1 \leq p < \infty$.

\begin{theorem} \label{thm:reconstructLp}
  Let $\sigma \in \LpRtd, 1 \leq p < \infty$,
  $\supp \hat{\sigma} \subseteq Q = [-\frac{1}{2a},\frac{1}{2a}]^d \times [-\frac{1}{2b},\frac{1}{2b}]^d$ and 
  $g \in \LiRd \cap \mathcal{F}\LiRd$. Choose $\varphi \in \mathcal{C}_c^\infty(\Rtd)$ such that  $\varphi = 1$ on $Q$, define
  $K = \mathcal{F}^{-1} \left( \frac{\varphi}{\;\overline{\mathcal{U} V_g g}\;} \right)$ and assume that $\overline{\mathcal{U} V_g g}$ 
  does not vanish on $\supp \varphi$. Then the symbol $\sigma$ can be reconstructed from the diagonal entries 
  $H_{\lambda, \lambda} = \langle \sigma^{KN} \pi(\lambda)g, \pi(\lambda)g \rangle$ of the channel matrix via the modified cardinal series
    \begin{equation} \label{equ:reconstructLp}
      \sigma = \frac{1}{(ab)^d} \sum_{\lambda \in \Lambda} H_{\lambda, \lambda} T_{\lambda} (\sinc \ast K).
    \end{equation}
  The sum converges absolutely, uniformly and in $\LpRtd $. 
  If $p = 1$, the reconstruction formula is still valid, but the
  series converges only in  $\textbf{\textit{L}}^q (\Rtd)$ for $ q > 1$, but	
  not  in $\textbf{\textit{L}}^1 (\Rtd)$.
\end{theorem}
\begin{proof}
 We apply the multivariate version of the classical
Shannon-Whittaker-Kotelnikov sampling theorem with the lattice
$\Lambda = a \ZZ ^d \times  b \ZZ ^d$ to the bandlimited function
$\sigma \ast R(g,g)^\ast $.

We recall from Lemma \ref{lem:diagelem} that $H_{\lambda,\lambda} =
\sigma \ast R(g,g)^\ast (\lambda)$ and write 
\begin{equation} \label{equ:sigmarggsinc}
  \sigma \ast R(g,g)^\ast = \frac{1}{(ab)^d} \sum_{\lambda \in \Lambda} H_{\lambda,\lambda} T_\lambda \sinc,
\end{equation}
According to the $\Lpsp$-theory of the cardinal
series~\cite{boche08,si92}   the sum converges absolutely, uniformly, 
and in $\LpRtd$ for $ 1 < p < \infty$. 
  This sampling expansion holds pointwise also for $p=1$, but the
  convergence is then only in $\textbf{\textit{L}}^q (\Rtd)$ for $ q >
  1$. 

Since $R(g,g)^\ast \in \LiRtd$, formula~\eqref{equ:rihaczekstft} implies that  $\overline{\mathcal{U} V_g g} = \widehat{R(g,g)^\ast}  \in \mathcal{F} \LiRtd$. 
Since $\overline{\mathcal{U} V_g g} \neq 0$ on $\supp \varphi$, we apply the Wiener-L\'evy theorem, see \cite[Theorem 1.3.1]{re68}, 
to conclude that there exists a function $\psi \in \mathcal{F} \LiRtd$ such that 
$\psi = \frac{1}{\;\overline{\mathcal{U} V_g g}\;}$ on $\supp \varphi$. Since $\varphi \in \mathcal{C}_c^\infty(\Rtd)$, $\varphi$ is 
also in $\mathcal{F} \LiRtd$ and we have $\varphi \,\psi = \frac{\varphi}{\;\overline{\mathcal{U} V_g g}\;} \in \mathcal{F} \LiRtd$. 
Thus we conclude that $K  = \mathcal{F}^{-1} \left(
  \frac{\varphi}{\;\overline{\mathcal{U} V_g g}\;} \right) \in \LiRtd$. 

Now we claim that
\begin{equation} \label{equ:deconvolvesigma}
  \sigma \ast R(g,g)^\ast \ast K = \sigma.
\end{equation}
To prove this, we take  the Fourier transform of
(\ref{equ:deconvolvesigma}) and obtain 
\begin{equation} \label{equ:deconvolvesigma2}
  \hat{\sigma} \cdot \widehat{R(g,g)^\ast} \cdot \hat{K} 
  = \hat{\sigma} \cdot \overline{\mathcal{U} V_g g} \cdot \overline{\mathcal{U} V_g g}^{-1} \cdot \varphi
  = \hat{\sigma}, \nonumber
\end{equation}
because $\varphi = 1$ on $\supp \hat{\sigma}$.

Finally we combine (\ref{equ:sigmarggsinc}) and (\ref{equ:deconvolvesigma}) to compute
\begin{equation}
  \sigma 
  = \sigma \ast R(g,g)^\ast \ast K
  = \frac{1}{(ab)^d} \sum_{\lambda \in \Lambda} H_{\lambda,\lambda} T_\lambda (\sinc \ast K). \label{equ:reconstructsigmaLp}
\end{equation}
Since we convolve $\sigma \ast R(g,g)^\ast \in \LpRtd$ with $K \in \LiRtd$, the series in (\ref{equ:reconstructsigmaLp}) 
inherits the convergence properties from (\ref{equ:sigmarggsinc}). 
\end{proof}

Similar reconstruction formulas are valid when one or several
side-diagonals of the channel matrix $H$  are
known~\cite[Chpt.4]{pa11}. In the engineering practice the diagonal
entries are estimated only on a sublattice $\Lambda _p \subseteq
\Lambda $ by means of pilot symbols~\cite{coerpuba02, best03}.  In this case one
applies Theorem~\ref{thm:reconstructLp} to a sublattice of the full
lattice. 

In the case $p = 2$ we can weaken the assumptions on $g$ considerably to $g \in \LtRd$. This case is important because it treats
the (unjustified) assumption of the engineering community that wireless channels are Hilbert-Schmidt operators \cite{kopf06,pfwa06}. 
In addition, it covers the rectangular window $g =
\chi_{[\alpha,\beta]^d}$ corresponding to OFDM without pulse-shaping  
\cite{coerpuba02, fema10, best03}.

\begin{proposition}
  With the notation of Theorem \ref{thm:reconstructLp} assume that $\sigma \in \LtRtd$ and $g \in \LtRd$. Then $\sigma$ can be
  reconstructed from $(H_{\lambda,\lambda})_{\lambda \in \Lambda}$ by 
  \begin{equation} \label{equ:reconstructL2}
    \sigma = \frac{1}{(ab)^d} \sum_{\lambda \in \Lambda} H_{\lambda, \lambda} T_{\lambda} (\sinc \ast K).
  \end{equation}
  with convergence in $\LtRtd$ and uniform convergence.
\end{proposition}
\begin{proof}
The proof requires only a minor modification. Since $\widehat{R(g,g)^\ast} = \overline{\mathcal{U} V_g g}$ is bounded, the product
$\hat{\sigma} \cdot \widehat{R(g,g)^\ast}$ is in $\LtRtd$ with support in $Q$. Thus $\sigma \ast R(g,g)^\ast \in \LtRtd$ is 
bandlimited and the sampling reconstruction (\ref{equ:sigmarggsinc}) holds with uniform convergence and convergence in $\LtRtd$.
  
Finally, since $\overline{\mathcal{U} V_g g}$ does not vanish on $\supp \varphi$ by assumption, the multiplier 
$\hat{K} = \varphi \cdot \overline{\mathcal{U} V_g g}^{-1}$ is bounded, and therefore the operator $F \mapsto F \ast K$ is bounded
on $\LtRtd$. Consequently, the deconvolution formulas (\ref{equ:deconvolvesigma}) and (\ref{equ:deconvolvesigma2}) are well-defined on
$\LtRtd$ and the reconstruction (\ref{equ:reconstructL2}) follows. 
\end{proof}

Finally we formulate a distributional version of the reconstruction theorem. This version is not just for the sake of mathematical
generalization, but is necessary for the accurate modelling of physical channels. For example, a single point scatterer with time delay 
$\tau$ and Doppler shift $\nu$ has the point measure $\delta_{(\tau,\nu)}$ as its spreading function. A typical spreading
function is usually written as a distributional part plus a random component \cite{hlma06}. By adapting the hypothesis of Theorem 
\ref{thm:reconstructLp} we obtain the following statement.
\begin{proposition} \label{prop:recsymbdist}
  Let $\sigma \in \sch'(\Rtd)$, $\supp \hat{\sigma} \subseteq 
  Q_\vareps = [-\frac{1}{2a} + \vareps,\frac{1}{2a} - \vareps]^d \times [-\frac{1}{2b} + \vareps,\frac{1}{2b} - \vareps]^d$ for 
  some $\vareps > 0$ and $g \in \sch(\RR^d)$. Choose $\varphi \in \mathcal{C}_c^\infty(\Rtd)$ such that  $\varphi = 1$ on $Q_\vareps$, 
  $\supp \varphi \subseteq [-\frac{1}{2a},\frac{1}{2a}]^d \times [-\frac{1}{2b},\frac{1}{2b}]^d$ and assume that 
  $\overline{\mathcal{U} V_g g}$ does not vanish on $\supp \varphi$. Then the symbol $\sigma$ can be reconstructed from the diagonal
  entries $H_{\lambda, \lambda} = \langle \sigma^{KN} \pi(\lambda)g, \pi(\lambda)g \rangle$ of the channel matrix via the 
  modified cardinal series
  \begin{equation} \label{equ:reconstructdistribution}
    \sigma = \frac{1}{(ab)^d} \sum_{\lambda \in \Lambda} H_{\lambda, \lambda} T_{\lambda} \mathcal{F}^{-1} 
    \left( \frac{\varphi}{\;\overline{\mathcal{U} V_g g}\;} \right),
  \end{equation}
  with distributional convergence.
\end{proposition}
\begin{proof}
  If $g \in \sch (\Rd)$, then $R(g,g) \in \sch (\Rtd)$ and thus $\sigma \ast R(g,g)^\ast \in \sch' (\Rtd)$ with 
  $\supp \widehat{\sigma \ast R(g,g)^\ast} \subseteq Q_\vareps$. The distributional version of the sampling theorem \cite{ca68} now yields that
  \begin{equation}
    \sigma \ast R(g,g)^\ast 
      = \frac{1}{(ab)^d} \sum_{\lambda \in \Lambda} (\sigma \ast R(g,g)^\ast)(\lambda) T_{\lambda} \mathcal{F}^{-1} \varphi. \nonumber
  \end{equation}
  with distributional convergence. Since $\sigma \ast R(g,g)^\ast$ is
  an entire function of at most polynomial growth on $\Rtd$ (by the 
  theorem of Paley-Wiener~\cite{rudin73}) the pointwise evaluations are well-defined. Likewise, since $\sigma^{KN}$ is continuous from $\sch (\Rd)$
  to $\sch' (\Rd)$, the mapping $\lambda \mapsto \langle \sigma^{KN} \pi(\lambda) g, \pi(\lambda) g \rangle$ is continuous, therefore,
  as in Lemma \ref{lem:diagelem}, $H_{\lambda,\lambda} = (\sigma \ast R(g,g)^\ast)(\lambda)$ for $\lambda \in \Lambda$.
  
  To conclude, we observe that $\varphi \cdot \overline{\mathcal{U}
    V_g g}^{-1}$ is in $\sch (\Rtd)$ and thus also 
  $K = \mathcal{F}^{-1} \left( \varphi \cdot \overline{\mathcal{U} V_g
      g}^{-1} \right) \in \sch (\Rtd)$. Consequently, the deconvolution
  formulas (\ref{equ:deconvolvesigma}) and (\ref{equ:deconvolvesigma2}) make sense in $\sch' (\Rtd)$, and the reconstruction formula is proved.
\end{proof}

Proposition~\ref{prop:recsymbdist} will also be relevant for the numerical
implementation of the reconstruction formula. By assuming a slightly
smaller spectrum $Q_\epsilon $, the expanding kernel $\mathcal{F}^{-1}
\left( \varphi \cdot \overline{\mathcal{U} V_g       g}^{-1}
\right)$ is in $\mathcal{S} (\Rtd ) $ and decays rapidly. If $\sigma \in
\LpRtd $ instead of $\mathcal{S}'(\Rd )$, then the expansion is
localized and converges rapidly.

\section{Uniqueness Results}
\label{sec:uniqueness}

The reconstruction results of the previous section imply that a \psdo\
with a bandlimited symbol  is uniquely determined by the diagonal of
the channel matrix. In this section, we prove further  uniqueness
results that illustrate the relation between a \psdo\ and the
corresponding channel matrix under various assumptions on the Gabor
system and the symbol. For notational simplicity, we  now work in dimension $d=1$.

Let $\Lambda = a\ZZ \times b\ZZ$ be a lattice, let $g \in \LtR$ and $\gabsys$ the corresponding 
Gabor system in $\LtR$. We  assume that
\begin{equation} \label{equ:matrixzero}
  H_{\lambda \mu} = \langle \sigma^{KN} \pi(\mu)g, \pi(\lambda)g \rangle = 
  0, \quad \forall \lambda, \mu \in \Lambda.
\end{equation}
Of course, if $\gabsys $ spans $\LtRd $, then obviously $\sigma ^{KN}
= 0$. Under the basic assumptions of wireless communications
(bandlimited symbol and Gabor Riesz sequence) the conclusion is not so
obvious. 

Before we state the next theorem, we rewrite the general entries of
the channel matrix. 
We write $\lambda = (\lambda_1,\lambda_2), \mu = (\mu_1,\mu_2) \in \RR^2$.	
From the definition of $\sigma^{KN}$, see (\ref{equ:kn2}), we have
\begin{equation}
  \langle \sigma^{KN} \pi(\mu)g, \pi(\lambda)g \rangle = 
  \langle \hat{\sigma}, \mathcal{U} V_{\pi(\mu)g}  \pi(\lambda)g \rangle. \nonumber
\end{equation}
Using the covariance property of the STFT (e.g.,  \cite[Ch.~3]{gr01}), we compute
\begin{eqnarray}
  V_{\pi(\mu)g}  \pi(\lambda)g (-x,\xi)
  &=& e^{-2 \pi i (\xi - \lambda_2) \lambda_1} e^{2 \pi i \mu_2 (-x - \lambda_1)} 
  T_{\lambda-\mu} V_g g (-x,\xi) \nonumber\\
  &=& e^{2\pi i \lambda_2 \lambda_1} e^{-2\pi i \mu_2 \lambda_1} M_{(\mu_2,-\lambda_1)} 
  T_{\lambda - \mu} V_g g (-x,\xi)\, . \nonumber
\end{eqnarray}
Writing  $G = \mathcal{U} V_g g$, we  conclude that (\ref{equ:matrixzero}) is equivalent to
\begin{equation} \label{equ:matrixzero2}
  \langle \hat{\sigma}, M_{(\mu_2,-\lambda_1)} T_{\lambda - \mu} G \rangle = 
  0, \quad \forall \lambda,\mu \in \Lambda,
\end{equation}
 Here $\mu_2 \in b\ZZ$ and $\lambda _1 \in a\ZZ$, so the modulations
are taken with respect to the lattice $\Lambda' = b\ZZ \times a \ZZ$. 
 Combining (\ref{equ:matrixzero}) and (\ref{equ:matrixzero2}), we obtain
\begin{equation} \label{equ:matrixzero3}
  \langle \hat{\sigma}, M_{\mu} T_{\lambda} G \rangle = 
  0, \quad \forall \lambda \in \Lambda, \mu \in \Lambda' = b\ZZ \times a\ZZ.
\end{equation}

The first uniqueness theorem treats the case of Gaussian Gabor
systems. 

\begin{theorem} \label{thm:matrixzerogauss}
  Let $\sigma \in \sch'(\RR^2)$, $\varphi(x) = e^{- \pi x^2}$ be the Gaussian function and 
  $\Lambda = a\ZZ \times b\ZZ$ an arbitrary lattice. If $\hat{\sigma}$ is compactly supported and
  \begin{equation} \label{equ:thmmatrixzero}
    H_{\lambda \mu} = \langle \sigma^{KN} \pi(\mu)\varphi, \pi(\lambda)\varphi \rangle = 
    0, \quad \forall \lambda, \mu \in \Lambda,
  \end{equation}
then $\sigma^{KN}$ is identically zero.
\end{theorem}
\begin{proof}
The STFT of the Gaussian function is given by \cite[Lemma 1.5.2]{gr01}
\begin{equation} \label{equ:stftgauss}
  G(\xi,x) = \mathcal{U} V_\varphi \varphi (\xi,x) = V_{\varphi} \varphi (-x,\xi) = 
  \frac{1}{\sqrt{2}} \;e^{-\frac{\pi}{2} \xi^2- \frac{\pi}{2} x^2 + \pi i \xi x}. \nonumber
\end{equation}
Setting $\mu = (\mu_1,\mu_2) \in \Lambda'$, $\lambda = (\lambda_1, \lambda_2) \in \Lambda$, $z = (\xi,x) \in \RR^2$, 
we obtain that
\begin{eqnarray}
  M_{\mu} T_{\lambda} G 
  &=& e^{2 \pi i z \cdot \mu} G(z - \lambda) \nonumber \\
  &=& \frac{1}{\sqrt{2}}\; e^{2 \pi i (\xi \mu_1+ x \mu_2) + \pi i  (\xi-\lambda_1)(x-\lambda_2) 
  - \frac{\pi}{2} (\xi-\lambda_1)^2 - \frac{\pi}{2} (x - \lambda_2)^2}. \label{equ:exponent}
\end{eqnarray}
The exponent in (\ref{equ:exponent}) can be written as 
\begin{eqnarray}
  &&2 \pi i (\xi \mu_1 + x \mu_2) + \pi i  (\xi-\lambda_1)(x-\lambda_2) - \frac{\pi}{2} (\xi-\lambda_1)^2 - \frac{\pi}{2} (x - \lambda_2)^2 \\
  &=& \pi i \lambda_1 \lambda_2 - \frac{\pi}{2} \lambda_1^2 - \frac{\pi}{2} \lambda_2^2 - \frac{\pi}{2} \xi^2 - \frac{\pi}{2} x^2 + \pi i \xi x - \nonumber \\
  &-& \pi i \xi \lambda_2 - \pi i x \lambda_1 + \pi \xi \lambda_1 + \pi x \lambda_2 + 2 \pi i (\xi \mu_1 + x \mu_2). \nonumber
\end{eqnarray}	   
Next we define
\begin{equation} \label{equ:Flambdadef}
  F_\lambda (\xi,x) = e^{- \pi i \xi \lambda_2 - \pi i x \lambda_1 + \pi \xi \lambda_1 + \pi x \lambda_2 }, \quad \lambda \in \Lambda
  \nonumber
\end{equation}
and
\begin{equation} \label{equ:rhodef}
  \rho = \hat{\sigma} \cdot \overline{G}.
\end{equation}
Since $G$ is in $\sch (\RR^2)$, $\rho$ is a tempered distribution with
compact support in a set $S \subseteq \RR^2$. Furthermore, since
$F_\lambda $ is infinitely differentiable on any bounded open set, $F_\lambda \rho $ is
again in $\sch ' (\RR^2)$ with compact support in $S$.   With this notation and  
$k_\lambda = e^{\pi i \lambda_1 \lambda_2 - \frac{\pi}{2} \lambda_1^2
  - \frac{\pi}{2} \lambda_2^2 }\neq 0$, the assumption (\ref{equ:thmmatrixzero}) can be recast as
\begin{eqnarray}
  \langle \hat{\sigma}, M_{\mu} T_{\lambda} G \rangle
  &=& k_\lambda  \langle \hat{\sigma}, M_{\mu} G F_\lambda \rangle 
  = k_\lambda  \langle \rho, M_{\mu} F_\lambda \rangle \nonumber \\
  &=& k_\lambda \left( \rho \cdot F_\lambda \right) \fthat (\mu) 
  = 0 \label{equ:ftrho}
\end{eqnarray}
for all $\lambda \in \Lambda$ and for all $\mu \in \Lambda'$.
  
In the  remainder of the proof we will apply  the Poisson summation formula
for compactly supported distributions \cite[Corollary 8.5.1]{fr98} for
each $\lambda \in \Lambda $.  
By  restricting the periodization to a specific open set, we will derive  a system of equations for the 
restrictions of $\rho$.  Then we will use the invertibility of a
Vandermonde matrix for some selected $\lambda$'s to conclude  
that all restrictions of $\rho$ vanish. 
 
By the Poisson summation formula for compactly supported distributions, (\ref{equ:ftrho}) is equivalent to
\begin{equation}
  \mathscr{P}_\lambda := \sum_{\nu \in \Lambda^\circ} T_\nu (\rho F_\lambda) = 0, \quad \forall \lambda\in \Lambda, \nonumber
\end{equation}
where $\Lambda^\circ = \frac{1}{b} \ZZ \times \frac{1}{a} \ZZ$ is the
adjoint lattice of $\Lambda$ and thus the dual lattice of $\Lambda
'$. 
Since $S$ is compact, there is a positive integer $L$ such that 
\begin{equation}
  S=\text{supp } \rho \subset R = \left(\frac{-L+1}{b},\frac{L-1}{b} \right) \times \left(\frac{-L+1}{a},\frac{L-1}{a} \right). \nonumber
\end{equation}
We define the \emph{open} rectangle $Q = \left(-\frac{1}{b},\frac{1}{b} \right) \times \left(-\frac{1}{a},\frac{1}{a}\right)$ and 
$Q_{j,k} = (\frac{j}{b},\frac{k}{a}) + Q $, for $-L \leq j,k \leq L$. 
Clearly  $Q$ contains a period of $\mathscr{P}_\lambda $ and 
\begin{equation} \label{equ:Rcovered}
  R \subset \left(\frac{-L}{b},\frac{L}{b} \right) \times \left(\frac{-L}{a},\frac{L}{a} \right) =  \bigcup_{-L+1 \leq j,k \leq L-1} Q_{j,k}. 
\end{equation}
Next we consider the restriction of the periodization $\mathscr{P}_\lambda$ to $Q$,
\begin{equation} \label{equ:periodizationrest}
  \mathscr{P}_\lambda \vert_Q = \sum_{\nu \in \Lambda^\circ} T_\nu (\rho F_\lambda) \vert_Q = 0. 
\end{equation}

Since $R$ contains $\text{supp } \rho$, (\ref{equ:periodizationrest}) reduces to
\begin{equation}  \label{equ:periodizationrest2}
  \mathscr{P}_\lambda \vert_Q 
  = \sum_{j = -L}^L \sum_{k = -L}^L T_{(\frac{j}{b},\frac{k}{a})} (\rho F_\lambda) \vert_Q 
  = \sum_{j = -L}^L \sum_{k = -L}^L T_{(\frac{j}{b},\frac{k}{a})} \rho
  \vert_Q \cdot T_{(\frac{j}{b},\frac{k}{a})}  F_\lambda \vert_Q
  = 0
\end{equation}
for all $\lambda \in \Lambda$.
In (\ref{equ:periodizationrest2}) the  sum is from $-L$ to $L$,
because this covers all shifts of $Q$ which  
 intersect  $\text{supp } \rho$ according to (\ref{equ:Rcovered}). 
These identities (as well as the subsequent derivations) are to be
understood  in the weak sense. For instance,
\eqref{equ:periodizationrest2} means that  for all $\varphi \in \sch (\RR^2 )$
with $\mathrm{supp}\, \varphi \subseteq Q$  
$$
  \langle \mathscr{P}_\lambda \vert_Q, \varphi \rangle = 
   \sum_{j = -L}^L \sum_{k = -L}^L \langle T_{(\frac{j}{b},\frac{k}{a})} \rho
   \cdot T_{(\frac{j}{b},\frac{k}{a})}  F_\lambda , \varphi \rangle 
  = 0 \, .
$$
Next we define 
\begin{equation} \label{equ:defrhojk}
  \rho_{j,k} :=  T_{(\frac{j}{b},\frac{k}{a})} \rho \vert_Q \, . \nonumber
\end{equation}
Since 
\begin{equation} \label{equ:Flambdashifted}
  F_\lambda (\xi - j/b,x - k/a) = F_\lambda (\xi ,x) F_\lambda (-j/b,
  -k/a)   \,  
\end{equation}
and $F_\lambda (\xi , x) \neq 0$ for all $x,\xi \in \RR $, we may  divide (\ref{equ:periodizationrest2}) 
by the exponentials in $x$ and $\xi$ and obtain 
\begin{eqnarray} 
  &&  \sum_{j=-L}^L \sum_{k=-L}^L \rho_{j,k} \; F_\lambda (-j/b,-k/a)  \nonumber \\
  &=& \sum_{j=-L}^L \sum_{k=-L}^L \rho_{j,k} \; e^{l_1 (\pi i k - \pi j \frac{a}{b})} e^{l_2 (\pi i j - \pi k \frac{b}{a})} = 0
  	\nonumber
\end{eqnarray}
for all $\lambda = (al_1,bl_2)  \in \Lambda$. By splitting the sum over $k$ into an even and an odd part, we obtain
\begin{eqnarray}
  &&\sum_{j=-L}^L \sum_{k \text{ even}} \rho_{j,k} \; \left(e^{- \pi j \frac{a}{b}}\right)^{l_1} e^{l_2 (\pi i j - \pi k \frac{b}{a})} +   \nonumber \\
  &+& \sum_{j=-L}^L \sum_{k \text{ odd}} \rho_{j,k} \; \left(-e^{- \pi j \frac{a}{b}}\right)^{l_1} e^{l_2 (\pi i j - \pi k \frac{b}{a})} = 0.	\nonumber
\end{eqnarray}
This identity holds for all $l_1,l_2 \in \ZZ$.    We note that the
$(4L+2) \times (4L+2)$-matrix $V$  with entries 
$V_{l_1 r} = z_r^{l_1}, r = 1, \dots, 4L+2, l_1 = 0,1,\dots,4L+1$,
where $z_r = e^{\pi i (r-L-1) a/b}$ for $r = 1,\dots,2L+1$ and
$z_r = -e^{\pi i (r-3L-2) a/b}$ for $r = 2L+2,\dots,4L+2$
is a Vandermonde matrix based on the
$4L+2$ distinct nodes $\pm e^{\pi j a/b}$. Likewise the
$(2L+1)\times (2L+1)$-matrix  $W$ with entries 
  $W_{l_2 k} = e^{-l_2 \pi   k b/a}, k=-L, \dots , L, l_2  =
  0,\dots,2L$ is an invertible Vandermonde matrix. 

We conclude that, for every $l_2$ and for every $j = -L, \dots, L$,
\begin{equation} \label{equ:keven}
  \sum_{k \text{ even}} \rho_{j,k} \; e^{l_2 (\pi i j - \pi k \frac{b}{a})} = 0
\end{equation}
and
\begin{equation} \label{equ:kodd}
  \sum_{k \text{ odd}} \rho_{j,k} \; e^{l_2 (\pi i j - \pi k \frac{b}{a})} = 0.
\end{equation}
We divide the equations  (\ref{equ:keven}) and (\ref{equ:kodd}) by $e^{l_2 \pi i j}$ and add them to obtain
\begin{equation} \label{equ:vdmk}
  \sum_{k= -L}^L \rho_{j,k} \; \left( e^{- \pi k \frac{b}{a}}
  \right)^{l_2} =  0 
\end{equation}
for every $l_2$ and for every $j = -L, \dots, L$.
Since the coefficient matrix $W$ is invertible, 
we conclude that $\rho_{j,k} = 0$ for $-L \leq j,k \leq L$. 

Using \cite[Theorem 2.2.1]{ho03}, 
we conclude that $\rho = 0$. Since $\rho = \hat{\sigma } G $  by  \eqref{equ:rhodef} and $G$
is a Gaussian,   we arrive at  $\hat{\sigma} = 0$, as was to be shown.  

\end{proof}

{\bf Idea of an alternative proof of Theorem
    \ref{thm:matrixzerogauss}.}

We sketch an alternative proof where the argument is based on the fact that the
Gaussian $g$ is a (strictly) totally positive function. 
This means that for  two arbitrary  sequences of real numbers $x_1 <
x_2 < \dots < x_n$ and $y_1< y_2 < \dots < y_n$ the matrix
$$
\Big( g(x_j - y_k), j,k = 1, \dots , n\Big)
$$ 
has a strictly positive determinant and is thus
invertible. See~\cite{Schoenb1951a,SchoenbWhit1953a} for the  fundamental properties of totally
positive functions. 

It is profitable to use the Weyl calculus of \psdo s, which is
formulated by means of the (cross-) Wigner
distribution of $f,g\in \LtR$
$$ 
W(f,g)(x,\xi) =   \int_{\bR } f(x + \tfrac{t}{2})\overline{g(x -
  \tfrac{t}{2})}e^{-2\pi i \xi \cdot t}\,dt\, .$$
The Wigner distribution satisfies the following covariance property
(~\cite{folland89} or~\cite[Prop.~4.3.2c]{gr01}):
$$
W(\pi (\mu )\pi  (\lambda
  )f, \pi (\lambda )g) = e^{-\pi  \mu _1\mu _2} M_{\tilde \mu }
  T_\lambda W(f,g) \, ,
$$
where $\tilde \mu = (\mu _2, -\mu _1)$ for $\mu = (\mu _1,\mu _2) \in
\bR ^2$.  The Wigner distribution of the Gaussian $\varphi (t) = e^{-\pi t^2}$
is the Gaussian
$$
W(\varphi , \varphi )(x,\xi )  = \sqrt{2} e^{-\pi ( x^2 + \xi ^2)/2}
$$
with Fourier transform $ \Phi (\xi , x) = \widehat{W(\varphi, \varphi) }(\xi
,x) =  2^{-\frac{3}{2}} e^{- 2\pi (\xi ^2 + x^2)}$. 
Now assume that $\hat \sigma \in \LiRt$ and set  $\hat{\tau
}(\xi ,x) = e^{\pi i x\xi} \hat{\sigma }(\xi,x)$, then we have 
\begin{equation}
  \label{eq:c9}
  \langle \sigma ^{KN} \pi (\lambda )g, \pi (\mu )\pi  (\lambda
  )g\rangle = \langle \tau , W(\pi (\mu )\pi  (\lambda
  )g, \pi (\lambda )g\rangle  
\end{equation}
(See \cite{folland89} or \cite[Chpt.~14.3]{gr01} for the transition
between the Kohn-Nirenberg calculus and the Weyl calculus.)
Consequently, $\langle \sigma ^{KN} \pi (\nu )\varphi, \pi  (\lambda
  )\varphi \rangle = 0$ for all $\lambda , \nu \in \Lambda $ holds, \fif\ 
  \begin{align}
\langle \tau , W(\pi (\mu )\pi  (\lambda
  )\varphi, \pi (\lambda )\varphi\rangle  &=    e^{\pi  \mu _1\mu _2} \langle \tau
  , M_{\tilde \mu } T_\lambda W(\varphi , \varphi )\rangle \notag \\
& = e^{\pi  \mu _1\mu _2}
    e^{-2\pi i \tilde \mu \cdot \lambda }\, \langle \hat{\tau },
    M_{-\lambda } T_{\tilde \mu } \Phi \rangle  \notag \\
& = e^{\pi  \mu _1\mu _2}
    e^{-2\pi i \tilde \mu \cdot \lambda }\, \Big( \hat \tau \cdot
    T_{\tilde \mu } \Phi \Big) \, \widehat{}\,  (\lambda )  =  0     \label{eq:c10}
  \end{align}
for all $\lambda , \mu \in \Lambda $, or equivalently for all $\lambda
\in \Lambda $ and $\tilde \mu \in  \Lambda '$. 
By the Poisson summation formula \eqref{eq:c10} is equivalent to 
$$
\sum _{\nu \in \Lambda ^\circ } T_\nu (\hat \tau \,T_{\tilde \mu } \Phi
)(\xi, x) = \sum _{j,k=-L} ^L \hat \tau (\xi +
\frac{j}{a},x+\frac{k}{b}) e^{-\pi (\xi +
\frac{j}{a}- bl_2 )^2/2 + (x+\frac{k}{b} +a  l_1)^2/2} =0 
$$ 
for almost all $x,\xi \in [0,1/a]\times [0,1/b] $ and for all $\tilde
\mu  = (bl_2, -al_1) \in \Lambda ' $.   
Since Gaussian functions are totally positive, the matrix with entries  
$ e^{-\pi (\xi +
\frac{j}{a}- bl_2 )^2/2}$, $|j|\leq L, l_1 = 0, \dots, 2L $ is
invertible, and likewise the matrix $e^{-\pi  (x+\frac{k}{b} -
  al_1)^2)/2}$, $ |k|\leq L, l_2 = 0, \dots 2L $ is invertible. From this
we conclude that $\tau (\xi +
\frac{j}{a},x+\frac{k}{b})  = 0$   for all $j,k$ and almost  all $x,\xi $. 
Consequently $\hat \tau  = 0$ almost everywhere and thus $\hat \sigma
= 0$. \qed

Finally, we investigate an important assumption made in wireless
communications. To facilitate the inversion of the channel matrix
in~\eqref{eq:c4}, it is commonly assumed that the channel matrix is
a diagonal matrix~\cite{dmbssbook10, boduscsh10, grhahlmasc07, shhwdaka10}.  Clearly, if the Gabor system $\gabsys $  is a (Riesz) basis for
$\LtRd$, then there is a bijection between operators and channel
matrices, and thus there exist diagonal channel matrices.

However, if  the Gabor system is
a frame, then this assumption is never satisfied. 

\begin{theorem} 
	\label{thm:frame}
Let $\sigma^{KN}$ be a bounded operator on $\LtRd$, $\Lambda = \abZd$ be a lattice with $ab < 1$ and $g \in \LtRd$ such that $\gabsys$ is a frame for $\LtRd$.
If
\begin{equation} \label{equ:thmframe}
	\langle \sigma^{KN} \pi(\mu)g, \pi(\lambda)g \rangle = d_\mu \delta_{\lambda \mu}, \quad \forall \lambda, \mu \in \Lambda,
\end{equation}
then $\sigma^{KN}$ is identically zero.
\end{theorem}
\begin{proof}
It follows from (\ref{equ:thmframe}) that $\sigma^{KN} \pi(\mu)g$ is orthogonal to the linear span of the set 
$\{ \pi(\lambda)g: \lambda \in \Lambda, \lambda \neq \mu \}, \forall \mu \in \Lambda$. 
That is
\begin{equation}
  \sigma^{KN} \pi(\mu)g \perp \text{span}{\{ \pi(\lambda)g: \lambda \in \Lambda, \lambda \neq \mu \}}, \quad \forall \mu \in \Lambda. \nonumber
\end{equation}
Since $\{ \pi(\lambda)g: \lambda \in \Lambda \}$ is a frame for $\LtRd$, according to \cite[Lemma IX]{dusc52}, 
there are two possibilities for the set $\{ \pi(\lambda)g: \lambda \in \Lambda, \lambda \neq \mu \}$. Either
\begin{itemize}
  \item $\{ \pi(\lambda)g: \lambda \in \Lambda, \lambda \neq \mu \}$ is incomplete in $\LtRd$, or
  \item $\{ \pi(\lambda)g: \lambda \in \Lambda, \lambda \neq \mu \}$ is again a frame for $\LtRd$.
\end{itemize}
If the set $\{ \pi(\lambda)g: \lambda \in \Lambda, \lambda \neq \mu \}$ is incomplete in $\LtRd$, then the set $\{ \pi(\lambda)g: 
\lambda \in \Lambda \}$ has to be a Riesz basis for $\LtRd$. Otherwise it would not be possible to 
have an incomplete set after removing one single element. It follows from the density and duality theory of frames 
\cite{chdehe99, gr01, ja95} that in this case $ab = 1$. This contradicts the assumption $ab < 1$.

Therefore the set $\{ \pi(\lambda)g: \lambda \in \Lambda, \lambda \neq \mu \}$ is again a frame for $\LtRd$. This implies that the 
orthogonal complement of  $\{ \pi(\lambda)g: \lambda \in \Lambda, \lambda \neq \mu \}$ has to be equal to the set $\{ 0 \}$. That is,
\begin{equation}
  \sigma^{KN} \pi(\mu)g \perp \text{span}{\{ \pi(\lambda)g: \lambda \in \Lambda, \lambda \neq \mu \}}, \quad \forall \mu \in \Lambda, \nonumber
\end{equation}
is only possible if $\sigma^{KN} \pi(\mu)g = 0 $, $\forall \mu \in \Lambda$ and therefore $\sigma^{KN} \equiv 0 $.
\end{proof}

It seems that some fundamental algorithms of wireless communications in
nonstationary environments are based on the incorrect assumption that
the channel matrix is diagonal. Nevertheless the intuition of the
communication engineers is perfectly correct and can be supported by
rigorous mathematical results. Indeed, if the symbol $\sigma $ is
smooth and the pulse $g$ of the Gabor system 
possesses a minimal amount of \tf\ concentration, then the channel
matrix decays rapidly off its diagonal~\cite{gro06,Gr10}. Hence, from a
numerical point of view, the channel matrix can be approximated well by
a diagonal matrix. This idea was used for improved equalization
methods in \cite{grhahlmasvXX,HGM07}. 

\subsection*{Acknowledgements}

The second author wants to thank Jose Luis Romero for helpful comments.

\bibliographystyle{spmpsci}      

\end{document}